\documentclass[11pt]{amsart}
\usepackage{palatino}
\usepackage{amsfonts}
\usepackage{amssymb}
\usepackage{amscd}
\usepackage[breaklinks,bookmarksopen,bookmarksnumbered]{hyperref}

\usepackage{amsmath}

\usepackage{xypic}
\usepackage{amscd}
\newfont{\sheaf}{eusm10 scaled\magstep1}

\newtheorem{theorem}{Theorem}[section]
\newtheorem{proposition}[theorem]{Proposition}
\newtheorem{corollary}[theorem]{Corollary}
\newtheorem{lemma}[theorem]{Lemma}
\theoremstyle{definition}
\newtheorem{definition}[theorem]{Definition}
\newtheorem{remark}[theorem]{Remark}

\newtheorem{conjecture/question}[theorem]{Conjecture/Question}

\newtheorem{remark/definition}[theorem]{Remark/Definition}
\newtheorem{terminology/notation}[theorem]{Terminology/Notation}

\def\pic{\rm Pic}
\def\codim{\rm codim}

\def\Ker{\rm Ker}
\def\Supp{\rm Supp}
\def\Sing{\rm Sing}
\def\Aut{\rm Aut}

\def\Sim{Sym}

\def\rank{\rm rank}

\def\det{ \rm det}

 \title{Pl\"ucker forms and  the theta map}
 \author{Sonia Brivio }
\address{Universit\'a di Pavia, Dipartimento di Matematica, strada Ferrata 1 Pavia;  email: sonia.brivio@unipv.it}
 \author{Alessandro Verra }
\address{Universit\'a Roma Tre, Dipartimento di Matematica, L.go S. Leonardo Murialdo, 1 00146 Roma;  email: verra@mat.uniroma3.it}

\begin{document}
\maketitle
\begin{abstract} 

Let $SU_X(r,0)$ be the moduli space  of semistable vector bundles of rank r and trivial determinant over a smooth, irreducible, complex projective curve $X$.
The theta map $\theta_r: SU_X(r,0) \to \mathbb P^N$ is the rational map defined by the ample generator of ${\pic} \ SU_X(r,0)$.  The main result of the paper is that $\theta_r$ is generically
injective if $g >> r$ and $X$ is general. This partially answers the following conjecture proposed by Beauville: $\theta_r$ is  generically injective if $X$ is not hyperelliptic. The proof relies on the study of  the injectivity of the determinant map ${d_E}: {\wedge}^r H^0(E) \to H^0({\det} \ E)$, for a vector bundle $E$ on $X$, and on the reconstruction of the Grassmannian $G(r,rm)$ from a natural multilinear form associated to it, defined in the paper as the Pl\"ucker form. The method applies to other moduli spaces of vector bundles on a  projective variety $X$.

\end{abstract}

\section{ Introduction} \noindent
In this paper we introduce the elementary notion of Pl\"ucker form of a pair $(E,S)$, where $E$ is a vector bundle of rank $r$ on a smooth,
irreducible, complex projective variety $X$ and $S \subset H^0(E)$ is a subspace of dimension $rm$. Then we apply this notion to the study of the moduli space $SU_X(r,0)$ of semistable vector bundles of rank $r$ and trivial determinant on a curve $X$. Let
$$
\theta_r: SU_X(r,0) \to {\mathbb P}(H^0(\mathcal L)^*)
$$
be the so called theta map, defined by the ample generator $\mathcal L$ of ${\pic} \ SU_X(r,0)$,  \cite{key6ter} . Assume $X$ has genus $g$, we prove the following main result: \medskip \par \noindent
\underline {\bf Main Theorem} \it $\theta_r$ is generically injective if $X$ is general and $g >> r$. \rm  \medskip \par \noindent
The theorem gives a partial answer to the following conjecture, or optimistic speculation, proposed by Beauville in \cite{key2} 6.1: \medskip \par \noindent
\underline {\bf Speculation} \it  \ $\theta_r$ is generically injective if $X$ is not hyperelliptic. \rm \medskip \par \noindent
To put in perspective our result we briefly recall some open problems on $\theta_r$ and some known results, see \cite{key2}. A serious difficulty in the study of  $\theta_r$ is represented by its indeterminacy locus, which is quite unknown.  Raynaud bundles and few more constructions provide examples of points in this locus when $r > > 0$, cf. \cite{key9} and \cite{key10}. In particular,  there exists an integer $r(X) > 0$ such that $\theta_r$ is not a morphism as soon as $r > r(X)$.  As a matter of fact related to this situation, some basic questions are still unsolved. For instance:
\begin{itemize} \it
\item[$\circ$] is $\theta_r$ generically finite onto its image for any curve $X$?
\item[$\circ$] is $\theta_r$ an embedding if $r$ is very low and $X$ is general?
\item[$\circ$] compute $r(g) := min \ \lbrace r(X), \ \text { $X$ \it curve of genus $g$} \rbrace $.
\end{itemize} \par \noindent
On the side of known results only the case $r = 2$ is  well understood: $\theta_2$ is an embedding unless $X$ is hyperelliptic of genus $g \geq 3$, see \cite{key1}, \cite{key5}, \cite{key13}. Otherwise $\theta_2$ is a finite 2:1 cover of its image, \cite{key6o}.  For $r = 3$ it is conjectured that $\theta_3$ is a morphism and this is proved for $g \leq 3$, see  \cite{key2} 6.2 and \cite{key3} .  To complete the picture of  known results we have to mention the case of genus two. In this case $\theta_r$ is generically finite, see \cite{key3} and \cite{key6}. Moreover it is a morphism iff $r \leq 3$,  \cite{key8bis}.   \par \noindent \rm To prove our main theorem we  apply a more general method, working in principle for more moduli spaces of vector bundles over  a variety  $X$ of arbitrary dimension. Let us briefly describe it.  \par \noindent
 Assume $X$ is embedded in $\mathbb P^n$ and consider a pair $(E,S)$ such that: (i) $E$ is a vector bundle of rank $r$ on $X$, (ii) $S$ is a subspace of dimension $rm$ of $H^0(E)$, (iii) $ {\det} \ E \cong \mathcal O_X(1)$. Under suitable stability conditions  there exists a coarse moduli space $\mathcal S$ for $(E,S)$, see for instance  \cite{key8} for an account of this theory.  Let $p_i: X^m \to X$ be the $i$-th projection and let
$$
e_{S, E}: S \otimes \mathcal O_{X^m} \to \bigoplus_{i = 1, \dots , m} {p_i}^*E
$$
be the natural map induced by evaluating global sections. We will assume that $e_{E,S}$ is generically an isomorphism for general pairs $(E,S)$. For such a pair  the degeneracy scheme $\mathbb D_{E,S}$ of $e_{E,S}$ is a divisor in $X^m$, moreover
$$
\mathbb D_{E,S} \in   \mid \mathcal O_{X^m} (1, \dots, 1) \mid,
$$
where $\mathcal O_{X^m}(1, \dots,1) := p_1^* \mathcal O_X(1) \otimes \dots \otimes p_m^* \mathcal O_X(1)$. In this paper $\mathbb D_{E,S}$ is defined  as \it the Pl\"ucker form of $(E,S)$. \rm The construction of the Pl\"ucker form of $(E,S)$  defines a rational map
$$
\theta_{r,m}: \mathcal S \to  \mid \mathcal O_{X^m} (1, \dots, 1) \mid,
$$
sending the moduli point of $(E,S)$ to $\mathbb D_{E,S}$.  Assume $X = \mathbb G$, where $\mathbb G$ is the Pl\"ucker embedding of the Grassmannian $G(r,rm)$. Then consider the pair $({\mathcal U}^*, H)$, where ${\mathcal U}$ is the universal bundle of $\mathbb G$ and $H = H^0({\mathcal U}^*)$. In this case the Pl\"ucker form  of $({\mathcal U}^*, H)$ is the zero locus 
$$
\mathbb D_{\mathbb G} \in \mid \mathcal O_{\mathbb G^m}(1, \dots, 1) \mid
$$
of a natural multilinear form related to $\mathbb G$. More precisely $\mathbb G$ is embedded in ${\mathbb P}( \wedge^r V)$, where $V = H^*$, and $\mathbb D_{\mathbb G}$ is the zero locus of the map  
$$
d_{r,m}: (\wedge^r V )^m \to \wedge^{rm} V \cong \mathbf C,
$$
induced by the wedge product. In the first part of the paper we prove that $\mathbb G$ is uniquely reconstructed from $\mathbb D_{\mathbb G}$ as soon as $m \geq 3$.  We prove  that: \medskip \par \noindent
\underline{\bf Theorem} \it Let $m \geq 3$ and let $x \in {\mathbb P} (\wedge^r V)$, then $x \in \mathbb G$ iff 
the following conditions hold true: \par \noindent
(1) $(x, \dots , x) \in ({\mathbb P}(\wedge^r V))^m$ is a point of multiplicity $m-1$ for $\mathbb D_{\mathbb G}$, \par \noindent
(2) ${\Sing}_{m-1}(\mathbb D_{\mathbb G})$ has tangent space of maximal dimension at $(x, \dots,x)$.
\rm \medskip \par \noindent
It follows essentially from this result that  the previous map $\theta_{r,m}$ is generically injective, provided some suitable conditions are satisfied. \par \noindent Indeed let $(E,S)$ be a pair as above and let
$
g_{E,S}: X \to \mathbb G_{E,S}
$
be the classifying map in the Grassmannian $\mathbb G_{E,S}$ of $r$ dimensional subspaces of $S^*$. In section 4 we use the previous theorem to prove that: \medskip \par \noindent
 \underline {\bf Theorem} \it $\theta_{r,m}$ is generically injective under  the following assumptions:   \par \noindent
(1) ${\Aut}(X)$ is trivial and $m \geq 3$, \par \noindent
(2) $g_{E,S}$ is a morphism birational onto its image, \par \noindent
(3) the determinant map $d_{E,S} : \wedge^r S \to H^0(\mathcal O_X(1))$ is injective. 
\rm \medskip \par \noindent
However the main emphasis of this paper is on the case where $X \subset \mathbb P^n$ is a general curve of genus $g$ and  $\mathcal O_X(1)$ has degree $r(m + g - 1)$.  Assuming this, we consider
the moduli space $\mathcal S_r$ of pairs $(E,H^0(E))$, where $E$ is a stable vector bundle of determinant $\mathcal O_X(1)$ and $h^1(E) = 0$. Let $t$ be an $r$-root of $\mathcal O_X(1)$, then $\mathcal S_r$ is birational to $SU_X(r,0)$ via the map
$$
\alpha: \mathcal S_r \to SU_X(r,0),
$$
sending the moduli point of $(E,H^0(E))$ to the moduli point of $E(-t)$.  In the second half of the paper we prove  that 
$$ \theta_{r,m} \circ {\alpha}^{-1} = \beta \circ \theta_r, $$  where $\theta_r$ is the theta map of $SU_X(r,0)$ and $\beta$ is a rational map. Moreover we prove that the assumptions of the latter theorem are satisfied if $X$ is general of genus $g >> r$.  \it Then it follows that $\theta_r$ is generically injective as soon as $X$ is general of genus $g >> r$.  \rm \par \noindent This completes the description of the proof of the main theorem of this paper.   It seems interesting to use Pl\"ucker forms for further applications.   

\section {Pl\"ucker forms} \noindent
Let $V$ be a complex vector space of positive dimension $rm$  and let ${\wedge}^rV$ be 
the {\it $r$}-exterior power of $V$.  On  ${\wedge}^r V$ we consider the multilinear form
\begin{equation}
d_{r,m}: (\wedge^r V)^m \to \wedge^{rm} V \simeq \mathbb C, 
\end{equation}
such that  $$ d_{r,m}(w_1, \dots, w_m) := w_1 \wedge \dots \wedge w_m. $$ 
Notice that $d_{r,m}$ is symmetric if $r$ is even and skew symmetric if $r$ is odd. 
We fix $m$ copies $V_1, \dots , V_m$ of $V$ and the spaces $\mathbb P_s := \mathbb P(\wedge^r V_s)$, $s = 1, \dots , m$,  of dimension
$N := \binom{rm}r -1$. Then we consider the Segre embedding
$$
\mathbb P_1 \times \dots \times  \mathbb P_m \hookrightarrow \mathbb P^{(N+1)m-1}
$$
and its projections $\pi_s: \mathbb P_1 \times \dots \times \mathbb P_m \to \mathbb P_s$, $s= 1, \dots ,  m$. The form $d_{r,m}$ defines the following hyperplane section of $\mathbb P_1 \times \dots \times \mathbb P_m$:
\begin{equation}
\mathbb D_{r,m} := \{ (w_1, \dots, w_m) \in {\mathbb P}_1 \times ...\times {\mathbb P}_m   \ \vert \   d_{r,m}(w_1, \dots, w_m) = 0 \ \rbrace.
\end{equation}
 \begin{definition} $\mathbb D_{r,m}$ is the Pl\"ucker form of ${\mathbb P}(\wedge^r V)^m$.
 \end{definition}
\rm \par \noindent 
$\mathbb D_{r,m}$ is an element of the linear system $\mid \mathcal O_{\mathbb P_1 \times \dots \times \mathbb P_m}(1, \dots , 1) \mid$, where
$$\mathcal O_{\mathbb P_1 \times \dots \times \mathbb P_m}(1, \dots , 1) = {\pi}_1^* O_{{\mathbb P}_1}(1) \otimes ... \otimes {\pi}_m^* O_{{\mathbb P}_m}(1).$$
Let $e_1, \dots ,e_{rm}$ be a basis of $V$ and let $\mathcal I$ be the set of all naturally ordered sets $I :=$
$ i_1 < \dots < i_r $ of integers in $[1, rm]$. We fix in $\wedge^r V_s$ the basis 
$$
e^{(s)}_I:= e_{i_1} \wedge \dots \wedge e_{i_r}, \ \ I = i_1 < \dots < i_r \in \mathcal I.
$$
Then any vector of $\wedge^r V_s$ is of the form $ \sum p^{(s)}_{_I} e^{(s)}_{_{I}}$,  where the coefficients $p^{(s)}_I$ are
the standard Pl\"ucker coordinates on $\mathbb P_s$.  This implies that
$$
d_{r,m}(w_1,  \dots, w_m) = \sum_{I_1 \cup \dots \cup I_m = \lbrace 1, \dots ,rm \rbrace}  p^{(1)}_{_{I_1}} \cdots p^{(m)}_{_{I_m}} e^{(1)}_{_{I_1}} \wedge \dots \wedge e^{(m)}_{_{I_m}}
$$
for each $(w_1, \dots, w_m) \in (\wedge^r V)^m$. Note that, to give a decomposition $$ I_1 \cup \dots \cup I_m = \lbrace 1, \dots ,rm \rbrace $$ as above, is equivalent to give a permutation
$
\sigma: \lbrace 1, \dots ,rm \rbrace \to \lbrace 1, \dots ,rm \rbrace
$ which is strictly increasing on each of  the intervals $$ \mathbb U_1 := [1,r], \ \mathbb U_2 := Ê[r+1,2r], \ \dots,\ \mathbb U_m := [(m-1)r+1,mr]. $$ Let $\mathcal P$ be the set of these permutations, then we conclude that
$$
 d_{r,m}(w_1, \dots, w_m) = \sum_{\sigma \in \mathcal P} sgn(\sigma) p^{(1)}_{_{\sigma(\mathbb U_1)}} \cdots p^{(m)}_{_{\sigma(\mathbb U_M)}} e_1 \wedge \dots \wedge e_{rm}.
$$
\hfill\par\noindent
Assume that $w := (w_1, \dots, w_m) \in {(\wedge^r V)}^m$ is a vector defining the point $o \in \mathbb P_1 \times \dots \times \mathbb P_m$, we want to compute the  Taylor series  of $\mathbb D_{r,m}$ at $o$.
Let  $t := (t_1, \dots, t_m) \in (\wedge^r V)^m$, then we have the identity
 $$
d_{r,m}(w_1 + \epsilon \ t_1, \dots, w_m + \epsilon \ t_m) =  \sum_{k = 0 \dots m}\partial^{m-k}_w d_{r,m}(t) \epsilon^k.
$$
We will say that the function  
$$
\partial^{m-k}_w d_{r,m}: (\wedge^r V)^m \to \mathbf C,
$$
sending $t$ to the coefficient $\partial^{m-k}_wd_{r,m}(t)$ of $\epsilon^k$, is the \it $k$-th polar of $d_{r,m}$ at $w$, \rm cf. \cite{key6'}.  Let $S := s_1 < \dots < s_k$ be a strictly increasing sequence of $k$ elements of  $M := \lbrace 1, \dots , m \rbrace$. We
will put  $k := \mid S \mid$. Moreover, for $ w = (w_1, \dots, w_m) \in (\wedge^r V)^m$, we define
$ w_{_{S }} := w_{s_1} \wedge \dots \wedge w_{s_k}.$ Note that  $\partial^0_w(t) = d(w_1, \dots,w_m)$ for each $t$. If $m-k \geq 1$ it turns out that
\begin{equation}
\partial^{m-k}_w d_{r,m}(t) = \sum_{|S| = k} sgn(\sigma_S) w_{_{M-S}} \wedge  t_{_{S}},
\end{equation}
where $\sigma_S: M \to M$ is  the permutation $(1, \dots, m) \to (j_1, \dots, j_{m-k}, s_1, \dots, s_k)$ such that $S = s_1 < \dots < s_k$ and $j_1< \dots < j_{m-k}$. 
\begin{definition} Let $W := {\wedge}^rV$ then
$$
q: \mathbb P(W^m) \to {\mathbb P}_1 \times ... \times {\mathbb P}_m
$$
is the rational map  sending the point defined by the vector $w = (w_1, \dots  ,w_m)$ of $W^m$ to the $m$-tuple of points defined by the vectors $w_1, \dots, w_m $. \end{definition} \noindent
Note that the pull-back of $d_{r,m}$ by $q$ is a homogeneous polynomial
$$
q^*d_{r,m} \in {\Sim}^m W^* = H^0(\mathcal O_{_{\mathbb P(W)}}(m)).
$$
We mention, without its non difficult proof, the following result
\begin{proposition} $\partial^{m-k}_w(d_{r,m})$ is the $k$-th polar form at $w$ of $q^*d_{r,m}$. \end{proposition} 
\rm \par \noindent
Let $\hat o \in \mathbb P(W^m)$ be the point defined by $ w = (w_1, \dots, w_m)$ and let $o = q(\hat o)$. For  the tangent spaces to $\mathbb P(W^m)$ at $\hat o$ and to $\mathbb P_1 \times \dots \times  \mathbb P_m$ at $o$ one has
\begin{itemize}
\item[$\circ$] $
T_{_{\mathbb P(W^m), \hat o}} = W^m /  \langle \ w \ \rangle  $
\item[$\circ$] $
T_{_{\mathbb P_1 \times \dots \times \mathbb P_m, o}} = W / \langle \ w_1 \ \rangle  \oplus \dots \oplus W / \langle \ w_m \ \rangle .
$
\end{itemize} Moreover the tangent map
$$
dq_{\hat o}: W^m /  \langle  \ w \ \rangle  \  \longrightarrow \ W / \langle \ w_1 \ \rangle  \oplus \dots \oplus W / \langle \ w_m \ \rangle 
$$
is exactly the map sending 
$$(t_1, \dots,t_m) \mod  \langle \ w \ \rangle  \to (t_1 \mod  \langle \ w_1 \ \rangle , \dots, t_m \mod  \langle \ w_m \ \rangle ).$$
 In particular we have
$$ 
{\Ker} \  dq_{\hat o} = \lbrace (c_1w_1, \dots, c_mw_m), \ (c_1, \dots, c_m) \in  \mathbb C^m \rbrace /  \langle \ w \ \rangle  .
$$
We can now use $dq_{\hat o}$ to study some properties of ${\Sing}(\mathbb D_{r,m})$. We consider  the $k$-osculating tangent cone $\mathcal C^k_o \subset T_{_{\mathbb P_1 \times \dots \times \mathbb P_m , o}}$ to $\mathbb D_{r,m}$ at $o$. 
\begin{lemma} Keeping the above notations one has: \begin{enumerate} 
\item[(1)] $ {\Sing}_k( \mathbb D_{r,m}) = \{ o \in \mathbb D_{r,m} \  \vert \ \  \partial^{m-i}_w(d_{r,m}) = 0, \  i \leq k-1 \ \}.$ \medskip
\item[(2)] 
$ \mathcal C^k_o =  dq_{\hat o}( \{ t \in W^m \mod  \langle \ w \ \rangle    \ \vert \  \partial^{m-i}_w (d_{r,m})(t)  = 0, \ i \leq k \}) $  
 \end{enumerate}
\end{lemma}
\begin{proof} By the previous description of $dq_{\hat o}$ any one dimensional subspace  $ l $ of  $T_{\mathbb P_1 \times \dots \times \mathbb P_m, o} $ is the isomorphic image by $dq_{\hat o}$ of the tangent space at $\hat o$ to an affine line 
$$L_t := \lbrace w + \epsilon t  \  \mid  \  \epsilon \in \mathbb C \rbrace \subset  \mathbb P(W^m),$$
 for some $t = (t_1, \dots, t_m) \in W^m$.   On the other hand the 
pull-back of  the Taylor series of $\mathbb D_{r,m}$ to $L_t$ is
$$
d_{r,m}(w + \epsilon \ t) = \sum_{i = 0 \dots m}  \partial^{m-i}_w(d_{r,m})(t) \epsilon^i,
$$
this implies (1) and (2).  \end{proof} \par \noindent Let $o \in \mathbb P_1 \times \dots \times \mathbb P_m$ be the point defined by the vector $(w_1, \dots, w_m)$  and let $v \in T_{\mathbb P_1 \times \dots \times \mathbb P_m, o}$ be a tangent vector to an arc of curve
$$ \lbrace w_1+ \epsilon \ t_1, \dots, w_m+ \epsilon \ t_m, \ \epsilon \in \mathbf C \rbrace. $$
Applying the lemma and the equality (3),   it follows:
\begin{theorem}  \par \noindent
\begin{enumerate}
\item{} $ o \in {\Sing}_k({\mathbb D}_{r,m})$ \  $\Longleftrightarrow$  \ $w_S = 0,$  \ $\forall S \in \mathcal I $,  \
 $\mid S \mid = m-k+1$.
\medskip \par \noindent
\item{} $ v $ is tangent to ${\Sing}_k (\mathbb D_{r,m})$ at $o$ iff  
$$ \sum_{s \in S} sgn(\sigma_s) w_{S-\lbrace s \rbrace} \wedge  t_s = 0, \ \forall S \in \mathcal I,  \ \mid S \mid = m-k+1,$$
\end{enumerate}
where $\sigma_s$ is  the permutation of $S$ shifting $s$ to the bottom and keeping the natural order in $S - s$.

\end{theorem} \noindent 

  \begin{proof} (i) By Lemma 2.4 (1),  $o \in {\Sing}_k(\mathbb D_{r,m})$ iff the $i$-th polar $\partial^i_w(d_{r,m})$ is zero for $i \leq k-1$. This is equivalent
to $w_S = 0$ for $\mid S \mid = m-k+1$. (ii) As above, consider a tangent vector $v$ at $o$ to the arc of curve $ \lbrace w_1+ \epsilon \ t_1, \dots, w_m+ \epsilon \ t_m, \ \epsilon \in \mathbf C \rbrace.$ By lemma 2.4 (2),  $v$ is tangent to ${\Sing}_k( \mathbb D_{r,m})$ at $o$ iff the coefficient of $\epsilon$ in  $(w + \epsilon t)_S$ is zero, $ \forall \ \mid S \mid = m - k + 1$. This is equivalent to
the condition $\sum_{s \in S} sgn(\sigma_s) w_{S-\lbrace s \rbrace} \wedge  t_s = 0, \ \forall S \in \mathcal I,  \ \mid S \mid = m-k+1$.
\end{proof}
\begin{corollary} The Pl\"ucker form $\mathbb D_{r,m}$ has no point of multiplicity $\geq m$.
 \end{corollary}
\begin{proof} Assume $\mathbb D_{r,m}$ has multiplicity $ \geq m$ at $o$. Then $w_S = 0$, \ $\forall S$ with $|S| = 1$. This  means $w_1 = \dots = w_m = 0$, which is impossible.
 \end{proof} \rm \par \noindent
We are specially interested to the behaviour of $\mathbb D_{r,m}$ along its intersection with the diagonal
\begin{equation}
\Delta \subset \mathbb P_1 \times \dots \times \mathbb P_m \subset \mathbb P^{(N+1)m - 1}.
\end{equation}
We recall that $\Delta$ spans the projectivized space of the symmetric tensors of $(\wedge^r V)^{\otimes m}$. Moreover, $\Delta$
is  the $m$-Veronese embedding of $\mathbb P(\wedge^r V)$. If $r$ is odd $d_{r,m}$ is skew symmetric and  $\mathbb D_{r,m}$ contains $\Delta$. If $r$ is even  then $$ \mathbb D_{r,m} \cdot \Delta $$ is an interesting hypersurface of degree $m$ in the projective space $\Delta$.  \par \noindent  Applying Theorem 2.5 to a point $o$ in the diagonal, we have:
 \begin{corollary} Let $o \in \Delta$. Then:\par \noindent\begin{enumerate}
\item{}  $ o \in {\Sing}_k(\mathbb D_{r,m})$  \ $\Longleftrightarrow$ \  $w^{\wedge m-k+1} = 0;$
\medskip \par \noindent
\item{}  $v \in T_{{\Sing}_k(\mathbb D_{r,m}), o}$ if and only if   
$$ 
\sum_{j \in S} sgn(\sigma_s) w^{\wedge(m-k)} \wedge t_{j} = 0, \ \ \forall S \in \mathcal I, \ \mid S \mid = m-k+1. 
$$
\end{enumerate}
\end{corollary} \rm \par \noindent
\begin{remark} \rm Let $o \in \Delta$ be as above, it follows from the corollary that: 
$$o \in \Delta \cap {\Sing}_{m-1}(\mathbb D_{r,m})  \Longleftrightarrow  w \wedge w = 0.$$ 
It is easy to see that $ \Delta \subset {\Sing}_{m-1}(\mathbb D_{r,m})$  if $r$ is odd. Let  $r$ be  even then 
$$ \mathbb G \subset  \Delta \cap  {\Sing}_{m-1}(\mathbb D_{r,m}), $$  
where $\mathbb G$  is the Pl\"ucker embedding in $\Delta = \mathbb P(\wedge^r V)$ of the Grassmannian  $G(r,V)$. However it is not  true that the equality holds in the latter case. In fact the equation  $w \wedge w = 0$ defines  $\mathbb G$ if and only if  $r=2$, see \cite{key7bis}. 
\end{remark}

\section{Pl\"ucker forms and Grassmannians} \noindent

In this section we will keep the notation $\mathbb G$ for the Pl\"ucker embedding of $G(r,V)$.  Our purpose is now to show that $\mathbb G$ is uniquely reconstructed from $\mathbb D_{r,m}$ and the diagonal $\Delta$. More precisely we will show the following:
\begin{theorem} Let $m \geq 3$, then
$$
\mathbb G = \lbrace o \in \Delta \cap  {\Sing}_{m-1}(\mathbb D_{r,m})   \  \mid \ {\dim} \ T_{{\Sing}_{m-1}(\mathbb D_{r,m}), o} \ \  is  \ \ maximal.  \  \rbrace
$$
\end{theorem} \medskip  \par \noindent
For the proof we need some preparation.  The following result of linear algebra will be useful: let $E$ be a vector space of dimension $d$ and let $w \in \wedge ^r E$ be a non zero vector. Consider the linear map
$$
\mu^s_w: \wedge^s E  \to \wedge^{r+s} E
$$
sending $t$ to $w \wedge t$. We have:
\begin{proposition} Let $d - 2r \geq s$, then $\mu^s_w$  has rank  $\geq \binom {d-r}s$ and the equality holds if and only if the vector $w$ is decomposable.
 \end{proposition} \begin{proof}  We fix, with the previous notations, a basis $\lbrace e_1, \dots , e_d \rbrace$ of $E$ and the corresponding basis $\lbrace e_I$, $I = i_1 < \dots < i_r \rbrace $ of $\wedge^r E$. Let 
 $e_{I_0}$ $:=$ $e_1 \wedge \dots \wedge e_r$ so that $I_0 = 1 < 2 < \dots < r$. Since $w$ is non zero we can assume that $w = e_{I_0} + \sum_{I \neq I_0} a_I e_I$. Let $W^-, W^+$ be the 
 subspaces of $E$ respectively generated by $\lbrace e_1, \dots , e_r \rbrace$ and $\lbrace e_{r+1}, \dots , e_d \rbrace$. Then we have the direct sum decomposition
 $$
 \wedge^{r+s} E =  E^+ \oplus E^-,
 $$
 where $E^+$ and $E^-$ are defined as follows: 
$$ E^+  = \lbrace e_{I_0} \wedge u, \ u \in \wedge^s W^+ \rbrace \ \rm and \it \ E^-   = \lbrace \sum_{i = 1, \dots , r} e_i \wedge v_i, \  v_i \in \wedge^{r+s-1} E \rbrace. $$ 
Let  $p^+:  \wedge^{r+s} E \to E^+$ be the projection map.
 Since $w = e_{I_0} + \sum_{I \neq I_0} a_I e_I$, the map $$ {(p^+ \circ \mu^s_w)}_{{\vert}_{\wedge^s W^+}} : \wedge^s W^+ \to E^+$$ is just the map $u \to e_{I_0} \wedge u$, in  particular it is an isomorphism. This implies that
 $$
  {\rank} \  \mu^s_w \geq {\rank} \  (p_r \circ \mu^s_w ) = {\dim} \  \wedge^s W^+ = \binom{d-s}r. 
 $$
 Let $w$ be \it decomposable, \rm then there is no restriction to assume $w = e_{I_o}$ and it follows 
${\dim} \  {\rm Im} \ \mu^s_w = \binom{d-r}s$. Now let's assume that $w$ is \it not decomposable. \rm To complete the proof it suffices to show that, in this case, \begin{equation} {\dim}  \  {\rm Im} \ \mu^s_w > \binom {d-s}r. \end{equation} By the above remarks $\mu^s_w$ is injective on $\wedge^s W^+$. Hence the   inequality (5) holds iff \begin{equation}
 \mu^s_w (\wedge^s W^+) \neq {\rm  Im} \ \mu^s_w.
 \end{equation}
 On the other hand $p^+ \circ \mu^s_w: \wedge^r W^+ \to E^+$ is an isomorphism and ${\dim} \ \wedge^s W^+ = \binom{d-r}s$. Therefore inequality (6) is satisfied iff there exists a vector $\tau \in \wedge^{r+s}E$ such that
\begin{equation}
0 \neq \tau  \in {\rm Im} \ \mu^s_w \cap {\Ker} \ p^+.
\end{equation}
 So, to  complete the proof, it remains to show  the following: \medskip \par \noindent
\underline {Claim} \it Let $d - 2r \geq s$ and $w$ be not decomposable. Then there exists a vector $\tau$ as above. \rm \medskip \par \noindent
\underline {Proof} By induction on $s$. If $s = 1$  we have ${\dim} \ {\rm Im} \ \mu^1_w \geq d - r$. It is proved in  \cite{key6bis} Prop. 6.27,  that  the strict inequality holds iff $w$ is not decomposable.
Hence we have ${\dim}  \ {\rm Im} \ {\mu}_{w}^1  > d - r$ and there exists a non zero $\tau \in {\rm Im} \ \mu^1_w \cap {\Ker} \ p^+.$
 \par\noindent
Now assume  that  ${\tau} \in \ {\rm  Im} \  {\mu}_{w}^{s-1}$ is a non zero vector satisfying the induction hypothesis. Let $N = \lbrace v \in E \ \mid \ \tau \wedge v = 0 \rbrace$. Then 
$N$ is the Kernel of the map $\mu^1_{\tau}: E \to \wedge^{r+s-1} E$ and, by the first part of the proof, $\rm dim \ \it N \leq r + s - 1$.  Since we are assuming $s + r \leq d -r$, it follows that we can find  a vector $e_k \in \lbrace e_1, \dots , e_d \rbrace $ such that \begin{equation}
e_k \wedge e_1 \wedge \dots \wedge e_r \neq 0 \ \rm and \it \ e_k \not\in N.
\end{equation} 
Then for such a vector we have $$
 0 \not=  e_k \wedge \tau = \sum b_{J} e_k \wedge e_J, \quad \vert J \vert = r+s-1, \quad I_0  \not\subset  \{ J \cup k  \}
 $$
and, moreover, $e_k \wedge \tau \in {\rm Im} \mu_{w}^s$. Hence the claim follows. 
\end{proof}
\medskip \par \noindent
From now on we will assume $m \geq 3$. Moreover  we identify $\wedge^r V$ to its image via the diagonal embedding
$$
\delta: \wedge^r V \to (\wedge^r V)^m,
$$
sending $w$ to $\delta(w) := (w,  \dots, w)$.  Let $o \in \Delta$ be the point defined by $w = (w, \dots, w)$. From  Corollary 2.7 (i), we have  that 
$$
\Delta \cap {\Sing}_{m-1}(\mathbb D_{r,m})  \ = \  \lbrace o \in \Delta \  \vert \  w \wedge w = 0 \rbrace.
$$
Moreover let $(t_1, \dots, t_m) \in (\wedge^r V)^m$, and let $v$ be a tangent vector at $o$ to 
$$  
\lbrace (w+\epsilon \ t_1, \dots, w + \epsilon \ t_m), \ \epsilon \in \mathbf C \rbrace \subset \mathbb P_1 \times \dots \times \mathbb P_m,
$$
 it follows from Corollary 2.7  that  $v$ is tangent to ${\Sing}_{m-1}(\mathbb D_{r,m})$ at $o$ iff 
$$
w \wedge t_j + t_i \wedge w = 0, \  \ 1 \leq i < j \leq m,
$$
in the vector space $\wedge^{2r} V$. 
Let
$$
\vartheta : (\wedge^r V)^m \to (\wedge^r V /  \langle \ w \ \rangle )^m
$$
be the natural quotient map, where   $(\wedge^r V /  \langle \ w \ \rangle )^m = T_{_{ \mathbb P_1 \times \dots \times \mathbb P_m, o}}$.
Consider
$$
T_{ o} := \lbrace (t_1, \dots, t_m)  \in (\wedge^r V)^m \  \mid  \  w \wedge  t_j + t_i \wedge w = 0, \  \ 1 \leq i < j \leq m \rbrace
$$
and note that, by the latter remark, one has
$$
{\vartheta}^{-1}( T_{_{{\Sing}_{m-1}(\mathbb D_{r,m}),o}}) = T_{o}.
$$
For any point $o \in \Delta \cap {\Sing}_{m-1}(\mathbb D_{r,m})$ we define
\begin{equation}
c_o = \ {\rm codimension \ of }  \ T_{_{{\Sing}_{m-1}(\mathbb D_{r,m}), o}} \  {\rm in}  \  T_{_{\mathbb P_1 \times \dots \times \mathbb P_m,o}}, 
\end{equation} \par \noindent
Since $\vartheta$ is surjective, it is clear that $c_o$ is the codimension of $T_{ o}$ in $(\wedge^r V)^m$.
\begin{lemma} Let $c_o$ be as above and let $B := \binom {(m-1)r}r$, then \begin{enumerate}
\item{} $c_o \geq mB$ if $r$ is even and $m \geq 3$,
\item{} $c_o \geq (m-1)B$ if $r$ is odd and $m \geq 3$,
\item{} $c_o = m-1$  if $m \leq 2$.
\end{enumerate}
 Moreover the equality holds in (i) and (ii)  iff $w$ is a decomposable vector. 
 \end{lemma}
 \begin{proof} Let $w^{\perp} \subset \wedge^r V$ be the orthogonal space of $w = (w, \dots, w)$ with respect to the bilinear form
 $$
 \wedge: \wedge^r V \times \wedge^r V \to \wedge^{2r} V.
 $$
 Moreover let $N \subset (\wedge^r V)^m$ be the subspace defined by the equations $$(-1)^r t_i + t_j = 0, \ 1 \leq i < j \leq m. $$ It is easy
 to check that
 $$
 T_{ o} = N + (w^{\perp})^m.
 $$
 Let $m \geq 3$ then $N$ is the diagonal subspace  if $r$ is odd and $N = (0)$ if $r$ is even. By Proposition 3.2, we have that ${\codim} \  w^{\perp}\  \geq B$ and moreover  the equality holds iff $w$ is a decomposable vector. This implies (i), (ii) and the latter
 statement. Let $m \leq 2$ then $N$ is either the diagonal subspace or the space of pairs $(t,-t), \ t \in \wedge^r V$. Arguing as above it follows that $c_o = (m-1)B$, i.e. $c_o = m-1$. This completes the proof.  \end{proof}
  \par \noindent
{\em \it Proof of Theorem 3.1} \par \noindent
The proof is now immediate: let $o \in \Delta \cap {\Sing}_{m-1}(\mathbb D_{r,m})$. It is obvious that the codimension $c_o$ is minimal
iff $dim \ T_{{\Sing}_{m-1}(\mathbb D_{r,m}),o}$ is maximal. Assume $m \geq 3$, by Lemma  3.3 $c_o$ is minimal iff $o \in \mathbb G$.  $\qed$ \medskip \par \noindent
Keeping our usual notations we have
$$
 {\mathbb G}^m \subset  \mathbb P_1 \times \dots \times \mathbb P_m \subset {\mathbb P}^{(N+1)^m -1},
 $$
 where the latter inclusion is the Segre embedding and $\mathbb G$ is the previous Pl\"ucker embedding. The restriction of $\mathbb D_{r,m}$ to ${\mathbb G}^m$  has a geometric interpretation given in the next lemma.   \par \noindent Let $o = (w_1, \dots, w_m) \in \mathbb G^m$. Then  we have $ w_s := v^{(s)}_1 \wedge \dots \wedge v^{(s)}_r ,$ where $v^{(s)}_1, \dots  , v^{(s)}_r \in V_s$ and  $s = 1, \dots , m$.  In particular $w_s$ is a decomposable vector, so it defines a point $l_s$ in $\mathbb G$. The vector space corresponding to $l_s$ is generated by the basis $v^{s}_{1}, \dots , v^{s}_{r}$. We will denote its projectivization by $ L_{s}$. 
\begin{lemma} The following conditions are equivalent: \begin{enumerate}
\item{} $o \in \mathbb D_{r,m}$,
\item{} $w_1 \wedge \dots \wedge w_m = 0$,
\item{}  $\{v_{i}^j \}$, $1 \leq i \leq r$, $1 \leq j \leq m$, is not a basis of $V$,
\item{} there exists a hyperplane in ${\mathbb P}(V)$ containing $L_1\cup \dots \cup L_m$.
\end{enumerate}
\end{lemma}
\begin{proof} Immediate. \end{proof}
 
 \par \noindent
\begin{lemma} ${\mathbb D}_{r,m} $ cuts on ${\mathbb G}^m$ an integral hyperplane section.
\end{lemma}
\begin{proof} Consider the correspondence
$$ I = \{ (l_1, \dots, l_m,H) \in {\mathbb G}^m \times {\mathbb P}(V^*) \  \vert \ L_1\cup \dots \cup L_m \subset H \},$$ and its  projections  $p_1 \colon I \to {\mathbb G}^m$ and $p_2 \colon I \to {\mathbb P}(V^*) $. Note that  the fibre of $p_2$ at any $H$ is the product of Grassmannians of $r-1$ spaces in $H$, which is  irreducible. Hence $I$ is irreducible. On the other hand we have
$ p_1(I) ={\mathbb D}_{r,m} \cap {\mathbb G}^m$ by Lemma 3.4 (iv). Hence the latter intersection is  irreducible. Since $O_{{\mathbb G}^m}(1)$ is not divisible in ${\pic}({\mathbb G}^m)$, it follows that $\mathbb D_{r,m} \cdot \mathbb G^m$ is integral.
\end{proof} \par \noindent
On $\mathbb G$ we consider the universal bundle $\mathcal U_r$. We recall that \it $\mathcal U_r$ is uniquely defined by its Chern classes, unless $m = 2$. \rm Let  $l \in \mathbb G$ and let  $L \subset {\mathbb P}(V)$ be the space corresponding to $l$. Then the fibre of $\mathcal U^*_r$ at $l$ is $H^0(\mathcal O_L(1))$,  moreover 
$ H^0({\mathcal U_{r}}^*)= V^* = H^0(O_{{\mathbb P}(V)}(1)) $. Let $\pi_s: {\mathbb G}^m \to {\mathbb G}$ be  the projection onto the $s$-th factor. On $\mathbb G^m$ we consider the vector bundle of rank $rm$
$$
{\mathcal F} \colon = \bigoplus_{s = 1, \dots , m} \pi_s^*  {\mathcal U}^*_r.
$$
For any point $o = (l_1, \dots, l_m) \in \mathbb G^m$, we have
$$
 {\mathcal F}_o = {({\mathcal U}^*_{r})}_{l_1} \oplus.. \oplus {({\mathcal U}^*_{r})}_{l_m} = H^0(O_{L_1}(1)) \oplus ... \oplus H^0(O_{L_m}(1)).
$$
In particular the natural evaluation map
\begin{equation}
{ev}^m: V^* \otimes \mathcal O_{{\mathbb G}^m } \to {\mathcal F},
\end{equation}
is a morphism of vector bundles of the same rank $rm$. 
\begin{definition} $\mathbb D_{\mathbb G}$ is the degeneracy locus of ${ev}^m$. \end{definition} \rm \par \noindent
\begin{theorem} $\mathbb D_{\mathbb G} = \mathbb D_{r,m} \cdot {\mathbb G}^m.$
\end{theorem} 
\begin{proof} Let $o = (l_1, \dots,l_m) \in \mathbb G^m$, then $ev^m_o$ is the natural restriction map
$$ 
 H^0(O_{{\mathbb P}(V)}(1)) \to H^0(O_{L_1}(1)) \oplus ... \oplus H^0(O_{L_m}(1)).
$$
Note that $ev^m_o$ is an isomorphism iff $L_1 \cup \dots \cup L_m$ is not  in a hyperplane of ${\mathbb P}(V)$. This implies  that $\mathbb D_{\mathbb G}$ is a divisor. Moreover $\mathbb D_{\mathbb G} = \mathbb D_{r,m} \cap \mathbb G^m$ by Lemma 3.4 and $ {\mathbb D}_{\mathbb G} \in \vert O_{{\mathbb G}^m}(1,..,1) \vert.$ Hence $\mathbb D_{\mathbb G} = \mathbb D_{r,m} \cdot {\mathbb G}^m.$ \end{proof} 
  \section{Pl\"ucker forms and moduli of vector bundles}
\noindent In this section we consider any integral, smooth projective variety $X \subset \mathbb P^n$
of dimension $d \geq 1$. We assume that $X$ is linearly normal and not degenerate. 
\begin{definition}  $(E,S)$ is a good pair on $X$ if
 \begin{enumerate} \it
\item{}  $E$ is a vector bundle of rank $r$ on $X$,
\item{} ${\det} \ E  \  \cong \mathcal O_X(1)$, 
\item{} $S \subset H^0(E)$ is a subspace of dimension $rm$, 
\item{}  $E$ is globally generated by $S$,
\item{} the classifying map of $(E,S)$ is a morphism birational onto its image.
\end{enumerate}\end{definition}  \par \noindent
Given $(E,S)$ we have the dual space $V := S^*$ and its Pl\"ucker form
$$
\mathbb D_{r,m} \subset \mathbb P(\wedge^r V)^m.
$$
We want to use it. Let us fix preliminarily some further notations: \begin{definition} \ \begin{enumerate} 
\item $\mathbb G_{E,S}$ is the Pl\"ucker embedding of the Grassmannian $G(r,V)$, 
\item $\mathcal U_{_{E,S}}$ is the universal bundle of $\mathbb G_{E,S}$, 
\item $d_{_{E,S}}: \wedge^r S \to H^0(\mathcal O_X(1))$ is the standard determinant map, 
\item $\lambda_{_{_{E,S}}}: \mathbb P^n \to {\mathbb P}({\wedge}^r V)$ is the projectivized dual of $d_{_{E,S}}$\ , 
\item $g_{_{E,S}}: X \to \mathbb G_{E,S}$ is the classifying map defined by $S$.
\end{enumerate}
\end{definition} \noindent
We recall that $g_{_{E,S} }$ associates to $x \in X$  the parameter point of the space ${\bf Im} \ ev_x^*$, where $ev: S \otimes \mathcal O_X \to E$ is the evaluation map. It
is well known that $g_{_{E,S}}$ is defined by the subspace ${\rm Im} \ d_{E,S}$ of $H^0(\mathcal O_X(1))$, in particular
$$
g_{_{E,S}} = {\lambda_{_{E,S}}}_{{\vert}_{\    X}}.
$$
Since $E$ is globally generated by $S$ and $g_{E,S}$ is  a birational morphism, the next three  lemmas describe standard properties.
\begin{lemma} One has $E \cong \lambda_{_{_{E,S}}}^* \mathcal U_{_{_{E,S}}}^*$ and $S = \lambda_{_{_{E,S}}}^* H^0(\mathcal U^*_{_{_{E,S}}})$ for
any good pair $(E,S)$. 
\end{lemma}  \noindent
We say that the good pairs $(E_1,S_1)$, $(E_2,S_2)$ are  \it isomorphic \rm if there exists an isomorphism $u: E_1 \to E_2$ such that $u^* S_1 = S_2$.  
\begin{lemma} Let $(E_1,S_1)$ and $(E_2,S_2)$ be good pairs. Then the following conditions are equivalent: \begin{enumerate}
 \item{}   $d_{_{_{E_1,S_1}}} = d_{_{E_2, S_2}} \circ (\wedge^r \alpha)$ for some isomorphism $\alpha: S_1 \to S_2$.  
\item{}  $f^*E_1 \cong E_2$ and $f^*S_1 = S_2$ for some automorphism $f \in {\Aut}(X)$.
\end{enumerate} \end{lemma} 
\begin{proof} (i) $\Rightarrow$ (ii) The projectivized dual of $\wedge^r \alpha$ induces an isomorphism $a: \mathbb G_{E_2,S_2} \to \mathbb G_{E_1,S_1}$ such that $g_{E_1,S_1} = a \circ  g_{E_2,S_2}$.
On the other hand,  $g_{_{E_i,S_i}}: X \to \mathbb G_{_{E_i,S_i}}$ is a morphism birational onto its image for $i = 1,2$. Hence $a$ lifts to an automorphism $f: X \to X$ with the required properties. (ii) $\Rightarrow$ (i) It suffices to put $\alpha = f^*$. \end{proof}
\medskip \par \noindent
Let $\rho_i: X^m \to X$ be the projection onto the $i$-th factor of $X^m$.  Then 
$$ 
{ev}_{_{E,S}}  \colon S \otimes O_{X^m} \to  \bigoplus_{i = 1, \dots , m} \rho_i^*E := \mathcal E
$$
is the morphism defined as follows. Let $U \subset  X^m$ be open, we observe that $\mathcal E(U) = E(U)^m$. Then we define the map
$
{ev}_{_{E,S}}(U): \ S  \to E(U)^m
$
as the natural restriction map. Since $ev_{_{E,S}}$ is a morphism of vector bundles of the same rank, its degeneracy locus is either $X^m$ or a divisor
$$ 
{\mathbb D}_{_{E,S}} \in \vert O_{X^m}(1,..,1) \vert.
$$
 \begin{definition} We will say that the divisor $\mathbb D_{_{E,S}}$ is the determinant divisor, or the Pl\"ucker form, of the pair $(E,S)$.  \end{definition} \par \noindent
 If the previous locus is $X^m$ we will say that $(E,S)$ has no Pl\"ucker form.
 \begin{lemma} Let $(E_1,S_1)$ and $(E_2,S_2)$ be isomorphic good pairs. Then $\mathbb D_{E_2,S_2} =  \mathbb D_{E_1,S_1}$.  \end{lemma}
\begin{proof} Let $u: E_1 \to E_2$ be an isomorphism such that $u^*S_2 = S_1$. Then, by taking the pull back of  $u$ to  $ev_{_{E_1,S_1}}: S_1 \otimes \mathcal O_X \to \mathcal E_1$,
we obtain $ev_{_{E_2, S_2}}$. This implies that $\mathbb D_{_{E_1, S_1}} = \mathbb D_{_{E_2,S_2}}$. \end{proof}
  \begin {remark} \rm Note that $\mathbb D_{E,S}$ contains  the multidiagonal $ \Delta_m$, i.e. the set of all the points $(x_1, \dots, x_m) \in X^m$ such that $x_i = x_j$ for some
 distinct $i,j \in \lbrace 1, \dots , m  \rbrace$. Moreover, $\Delta_m$ is a divisor in $X^m$  iff ${\dim} \ X = 1$. In this case $\mathbb D_{_{E,S}}$ is reducible:
\begin{proposition} Assume that $X$ is a curve, then
$$ 
{\mathbb D}_{_{E,S}} = (r+\epsilon) {\Delta}_m + \mathbb D^*_{_{E,S}}.
$$ 
where $\epsilon \geq 0$ and the support of the divisor $\mathbb  D^*_{_{E,S}}$ is the Zariski closure of the set
$$  \{ (x_1,..,x_m) \in X^m - {\Delta}_m \  \vert \ \exists \  s \in S \quad s(x_i) = 0, \ \ i= 1,..,m. \ \}$$
\end{proposition}
\begin{proof} Let  $x = (x_1,\dots, x_m) \in \Delta_m$. Then $ev_{_{E,S}}$ has rank $\leq rm -r$ at $x$. This implies that $x$ is a point of multiplicity $ \geq r$ of the determinant divisor ${\mathbb D}_{_{E,S}}$. Hence
$\Delta_m$ is a component of $\mathbb D_{_{E,S}}$ of multiplicity $\geq r$. This implies the statement.
\end{proof} \end{remark} \par \noindent
Actually, $\epsilon = 0$ if $E$ is a general semistable vector bundle on the curve $X$. It is enough to verify this property in the case $E = L^{\oplus r}$ and 
$S = H^0(E)$,  where  $L $ is a general line bundle on $X$ of degree $m + g -1$. In this case the Pl\"ucker form of $(E,S)$ is indeed $r$ times the Pl\"ucker
form of $(L, H^0(L))$. \par \noindent
It is  also non difficult to compute that $\mathbb D_{E,S} - r\Delta_m$ is 
numerically equivalent to $a^* r\Theta$, where $a: X^m \to {\pic}^m (X)$ is the natural Abel map and $\Theta \subset {\pic}^m (X)$ is a theta divisor.
Finally we consider the commutative diagram 
$$
\CD
{X^m} @>{g_{_{_{E,S}}}^m}>> {\mathbb G_{_{E,S}}^m} \\
@VVV @VVV \\
{(\mathbb P^{n})^m} @>{\lambda_{_{_{E,S}}}^m}>> { (\mathbb P^N)^m} \\
 \endCD
$$
where the vertical arrows are the inclusion maps.
 \begin{lemma} Let  $\mathbb D_{E,S}$ be the Pl\"ucker form of a good pair $(E,S)$, then
$$
 {\mathbb D}_{_{E,S}} = ( {\lambda_{_{_{E,S}}}^m})^*\mathbb D_{r,m}.
$$
 \end{lemma} 
\begin{proof} Lifting by $g_{_{E,S}}^m $ the map $ev^m: V \otimes \mathcal O_{\mathbb G_{_{E,S}}} \to \bigoplus_{i = 1, \dots , m}\pi_s^* \mathcal U^*_{_{E,S}}$, one obtains the map 
$  {ev}_{_{E,S}}  \colon S \otimes O_{X^m} \to  \bigoplus_{i = 1, \dots , m} \rho_i^*E.$ From  the commutativity of the above diagram it follows that 
 $\mathbb D_{_{E,S}} =  {({\lambda}_{E,S}^m)}^* \mathbb D_{r,m}
= {(g_{E,S}^m)}^* \mathbb D_{{\mathbb G}_{E,S}}$.  \end{proof} \par  \noindent 
To a good pair $(E,S)$ we have associated its Pl\"ucker form $\mathbb D_{_{E,S}}$.  Now we want to prove that, under suitable assumptions, a good pair $(E,S)$ is uniquely reconstructed from $\mathbb D_{_{E,S}}$.  To this purpose we define the following projective variety in  the ambient  space $\mathbb P^n$ of $X$.
  \begin{definition} \it $\Gamma_{_{E,S}}$ is the  closure of the set of points $x \in \mathbb P^n$ such that: \begin{enumerate}
\item $\mathbb D_{_{E,S}}$ has multiplicity $m-1$ at the point $o = (x, \dots,x) \in (\mathbb P^n)^m$,
\item the tangent space to ${\Sing}( \mathbb D_{_{E,S}})$ at $o$ has maximal dimension.
\end{enumerate}
\end{definition} \par \noindent
 \begin{theorem} Assume that $d_{_{E,S}}$ is injective and $m \geq 3$. 
Then: \begin{enumerate}
\item $\Gamma_{_{E,S}}$ is a cone in $\mathbb P^n$ with directrix the Grassmannian $\mathbb G_{_{E,S}}$, 
\item  the vertex of the cone $\Gamma_{_{E,S}}$ is the center of the  projection $\lambda_{_{_{E,S}}}$. 
\end{enumerate}
\end{theorem} 
\begin{proof}  Since $\lambda_{_{E,S}}$ is the projective dual of $d_{_{E,S}}$, the tensor product  map 
$$d_{_{E,S}}^{\otimes m}: (\wedge^r S)^{\otimes m} \to H^0(\mathcal O_X(1))^{\otimes m}$$
 is precisely the pull-back map
$$
(\lambda_{_{_{E,S}}}^m)^*: H^0(\mathcal O_{({\mathbb P}( \wedge^r V))^m}(1, \dots, 1)) \to H^0(\mathcal O_{(\mathbb P ^n)^m}(1, \dots, 1)).
$$
Moreover it is injective. Let $F \in  H^0(\mathcal O_{({\mathbb P}( \wedge^r V))^m}(1, \dots, 1)) $ be the  polynomial of multidegree $(1, \dots, 1)$ defining $\mathbb D_{r,m}$.  Then we can choose coordinates on  $({\mathbb P} (\wedge^r V))^m$ and $(\mathbb P^n)^m$ so that $d_{E,S}^{\otimes m}(F) = F$.  Assume that $\lambda_{E,S}^m$ is a morphism at the point $o \in (\mathbb P^n)^m$, then it follows that: 
\hfill\par\noindent  
(a) \it  $\lambda_{_{E,S}}^m(o) \in {\Sing}_{m-1}(\mathbb D_{r,m})$ iff $o \in {\Sing}_{m-1}(\mathbb D_{E,S})$, \rm
\hfill\par\noindent
(b) \it the codimension is equal for the  tangent spaces to ${\Sing}_{m-1}(\mathbb D_{r,m})$ at $\lambda_{_{E,S}}^m(o)$ and  to ${\Sing}_{m-1}(\mathbb D_{E,S})$ at $o$.
\rm \par \noindent
Assume that $o = (x, \dots, x)$ is a diagonal point in  $(\mathbb P^n)^m$. Then $x \in \Gamma_{E,S}$ iff $o$ satisfies  (i) and (ii) in Definition 4.10. By (a) and (b), conditions (i) and (ii) hold true for $o$ iff they hold true for $\lambda_{_{E,S}}^m(o)$ as a point of $\mathbb D_{r,m}$. Finally, by Theorem 3.1,  $\lambda_{_{E,S}}(o)$ satisfies (i) and (ii) iff $x$ belongs to the Grassmannian $\mathbb G_{E,S}$. Hence $\Gamma_{E,S}$ is a cone over $\mathbb G_{E,S}$ with vertex the center of $\lambda_{E,S}$. \end{proof}  \par \noindent
We are now able to show the main result of the current  section.  
\begin{theorem} Let $(E_{1}, S_{1})$ and  $(E_{2}, S_{2})$ be good pairs defining the same Pl\"ucker form $\mathbb D \subset (\mathbb P^n)^m$. Assume that  $m \geq 3$  and $d_{_{E_i,S_i}}$ is injective for any $i=1,2$, then  there exists $f \in {\Aut} (X)$
such that $f^{*}E_{2} \cong E_{1}$ and $f^{*}S_{2} = S_{1}$.
\end{theorem}
\begin{proof} Let $\Gamma$ be the closure of the set of diagonal points $o = (x, \dots, x) \in  \mathbb D$ of multiplicity $m-1$ and  tangent space $T_{{\Sing}_{_{m-1}}(\mathbb D), o}$ of maximal dimension.  By Theorem 4.11,  $\Gamma$  is a cone in $\mathbb P^n$: its directrix is the Grassmannian $G_{_{E_i,S_i}}$ and its vertex is the center of the projection $\lambda_{_{E_i,S_i}}$, both for $i = 1$
and $i = 2$. Since the projection maps   $\lambda_{_{E_i,S_i}}$ have  the same center, there exist an isomorphism  
$\sigma \colon G_{_{E_2,S_2}} \to G_{_{E_1,S_1}}$ such that 
$\lambda_{_{E_1,S_1}} = \sigma \circ \lambda_{_{E_2,S_2}}$.  Since $m \geq 3$, then $\sigma = {\wedge}^r {\alpha}^*$ for 
an isomorphism $\alpha \colon S_1 \to S_2$, see \cite{key7bis} p.122. 
 Then, applying Lemma 4.4, it follows $f^*E_1 \cong E_2$ and $f^*S_1 = S_2$ for some $f \in {\Aut}(X)$.
\end{proof}  \noindent
To conclude this section we briefly summarize,  in a general statement, how  to deduce from the previous results the generic injectivity of some natural maps, defined  on a moduli space of good pairs as above.  
Therefore we assume that a coarse moduli space $\mathcal S$ exists for the family of good pairs $(E,S)$ under consideration. This is, for instance the case when $E$ is stable with respect to the polarization
$\mathcal O_X(1)$ and $S = H^0(E)$. Then there exists a natural map
$$
\theta_{r,m}: \mathcal S \to \mid \mathcal O_{X^m}(1, \dots, 1) \mid
$$
sending the moduli point of $(E,S)$ to its determinant divisor $\mathbb D_{_{E,S}}$. Let $(E_1,S_1)$ and $(E_2,S_2)$ be good pairs as above defining two general points of $\mathcal S$. Assume
that $\mathbb D_{E_1,S_1} = \mathbb D_{E_2,S_2}$. Then we know from  Theorem 4.12 that then $(E_1,S_1)$ and $(E_2,S_2)$ are isomorphic if $m \geq 3$, ${\Aut} (X) = 1$ and
$$
d_{E_i,S_i}: \wedge^r S_i \to H^0(\mathcal O_X(1)).
$$
is injective. This implies the next statement:
\begin{theorem} Let $m \geq 3$ and ${\Aut} (X) = 1$. Assume 
$
d_{_{E,S}}: \wedge^r S \to H^0(\mathcal O_X(1))
$
is injective for good pairs $(E,S)$ with moduli in a dense open subset of $\mathcal S$. Then $\theta_{r,m}$ is generically injective.
\end{theorem}
 
 \section {Pl\"ucker forms and the theta map of $SU_X(r,0)$} \noindent
Now we apply the preceding arguments to study  the theta map of the moduli space $SU_X(r,0)$ of semistable vector bundles of rank $r$ and trivial determinant over a curve $X$ of genus $g \geq 2$. 
By definition the theta map
$$
\theta_r: SU_X(r,0) \to {\mathbb P}(H^0(\mathcal L)^*)
$$
is just the rational map defined by the ample generator $\mathcal L$ of $SU_X(r,0)$.  We prove our main result: 
\begin{theorem} Let $X$ be  general and $g >> r$, then $\theta_r$ is generically injective.
 \end{theorem} \par \noindent
To prove the theorem we need some preparation. At first we replace the space $SU_X(r,0)$ by a suitable translate of it, namely the moduli space
$$
\mathcal S_r
$$
of semistable vector bundles $E$ on $X$ having rank $r$ and fixed determinant $\mathcal O_X(1)$ of degree $r(m + g - 1)$.  We assume that $X$ has general moduli and that $\mathcal O_X(1)$ is general in ${\pic}^{r(m+g-1)} (X)$, with  $m \geq 3$ and $r \geq 2$. In particular $\mathcal O_X(1)$ is very ample: we also assume that $X$ is embedded in $\mathbb P^n$
by $ \mathcal O_X(1)$.  \par \noindent  We recall that $\mathcal S_r$ is biregular to $SU_X(r,0)$, the biregular  map being induced by  tensor product with an $r$-th root of $\mathcal O_X(-1)$.  \begin{proposition} Let $E$ be a semistable vector bundle on $X$ with general moduli in $\mathcal S_r$. Then: \begin{enumerate} \item $h^0(E)$ $= rm$  and
$(E,H^0(E))$ is a good pair,
\item the Pl\"ucker form of $(E, H^0(E))$ exists.
\end{enumerate}
\end{proposition} 
\begin{proof} (i) It suffices to produce \it one \rm semistable vector bundle $E$ on $X$,   of degree $r(m + g - 1)$ and rank $r$, such that $h^0(E) = rm$ and $(E,H^0(E))$ is a good pair in the sense of Definition 4.1. Then the statement follows because the conditions defining a good pair are open. Let $L \in {\pic}^{m + g - 1}(X)$ be general, then $h^0(L) = m$ and  $L$ is globally generated. Since $m \geq 3$,  $L$ defines a morphism birational onto its image $$f: X \to {\mathbb P}( H^0(L)^*).$$ Putting $E := L^{\oplus r}$ we have  a globally generated, semistable vector bundle such that $h^0(E) = rm$.  Hence, to prove that $(E, H^0(E))$  is a good pair,  it remains to show that its classifying map  
$$ g_E: X \to {\mathbb G}_E \colon = G(r, H^0(E)^*)$$ 
  is birational onto its image. We  observe  that $H^0 (E) = H_1 \oplus \dots 	\oplus H_r$, where $H_i$ is just a copy of $H^0(L), \ i = 1, \dots , r.$ Let $f_i: X \to {\mathbb P}(H^*_i)$ be the corresponding copy of $f$, for any $i =1, \cdots , r$.  Then $g_E: X \to {\mathbb G}_E$ can be described as follows: let $\mathbb P (E^*_x) \subset \mathbb P (H^0(E)^*)$  be the linear embedding induced by the evaluation map, it turns out that $\mathbb P(E^*_x)$ is the linear span of $f_1(x), \dots, f_r(x)$. This implies that $g_E = u \circ  (f_1 \times \dots \times f_r)$, where
$$
u: \mathbb P(H_1^*) \times \cdots \times \mathbb P(H_r^*) \to {\mathbb G}_E
$$
is the rational map sending $(y_1, \dots, y_r)$ to the linear span of the points $y_i \in \mathbb P(H_i^*) \subset \mathbb P(H^0(E)^*)$, $i = 1, \dots , r.$ Since $f$ is birational onto its image, the same is true for the map $f_1 \times \cdots \times f_r$. Moreover $u$ is clearly birational onto its image. Hence $g_E$ is birational onto its image. Finally $g_E$ is a morphism, since $L^{\oplus r}$ is globally generated. This completes the proof of (i). \par \noindent (ii) Again it suffices to produce \it one \rm good pair $(E,H^0(E))$ with the required property. It is easy to see that this is the case if $E = L^{\oplus r}$ as in (i). \end{proof} \par \noindent
Now we consider the rational map
$$
\theta_{r,m}: \mathcal S_r \to \mid \mathcal O_X(1, \dots, 1) \mid
$$
sending the moduli point  $[E] \in \mathcal S_r$ of a general $E$ to the Pl\"ucker form
$$
\mathbb D_E \in \mid \mathcal O_X(1, \dots, 1) \mid
$$
of the pair $(E, H^0(E))$. Let $t \in {\pic}^{m + g - 1}(X)$ be an $r$-root of $\mathcal O_X(1)$, then we have a map
$$
a_t: X^m \to {\pic}^{g-1}(X)
$$
sending $(x_1, \dots, x_m)$ to $\mathcal O_X(t - x_1 - \dots - x_m )$. It  is just the natural Abel map $a \colon X^m \to {\pic}^m(X)$, multiplied by $-1$ and composed with the tensor product by $t$. Fixing a Poincar\'e bundle $\mathcal P$ on $X \times {\pic}^{g-1}(X)$ we have the sheaf
$$
R^1 q_{2*}(q_1^*E(-t) \otimes \mathcal P),
$$
where $q_1, q_2$ are the natural projection maps of $X \times {\pic}^{g-1}(X)$. It is  well known the support of  this sheaf is either ${\pic}^{g-1}(X)$ or a Cartier divisor $\Theta_E$,  see \cite{key4}.  Moreover, due to the choice of $t$, one has
$$
\Theta_E \in \mid r\Theta \mid,
$$
where $\Theta := \lbrace N \in {\pic}^{g-1}(X) \ \mid \ h^0(N) \geq 1 \rbrace$ is the natural theta divisor of ${\pic}^{g-1}(X)$. In particular, one has $h^0(E \otimes N(-t)) = h^1(E \otimes N(-t))$ so  that
$$
{\Supp} \ \Theta_E = \lbrace N \in {\pic}^{g-1}(X) \ \mid \ h^0(E \otimes N(-t)) \geq 1 \rbrace.
$$
Finally, it is well known that there exists a suitable identification 
$$
\mid r\Theta \mid = \mathbb P(H^0(\mathcal L)^*)
$$
such that $\theta_r([E]) = \Theta_E$, \cite{key4}. Computing Chern classes it follows 
$$
a_t^* \Theta_E + r\Delta_m \in \mid \mathcal O_{X^m}(1, \dots, 1) \mid,
$$
where $\Delta_m \subset X^m$ is the multidiagonal divisor. On the other hand, $r \Delta_m$ is a component of
$\mathbb D_E$ by Proposition 4.8. Moreover, it follows from the definition of determinant divisor that $\mathbb D_E$
contains $a_t^{-1}(\Theta_E)$. Therefore we have \begin{equation}
a_t^* \Theta_E + r\Delta_m = \mathbb D_E.
\end{equation}
Let
$
\alpha: \mid r\Theta \mid \to \mid \mathcal O_{X^m}(1, \dots, 1) \mid
$
be the linear map sending $D \in \mid r\Theta \mid$ to $a_t^* D + r\Delta_m$. We  conclude from the latter equality that

\begin{proposition} $\theta_{r,m}$ factors through the theta map $\theta_r$, that is  \ $\theta_{r,m} = \alpha \circ  \theta_r$.
 \end{proposition} \medskip  \par \noindent
\begin{proof} [\underline {Proof of Theorem 5.1}] \ \par \noindent  Let $\theta_{r,m}: \mathcal S_r \to \mid \mathcal O_{X^m}(1, \dots, 1) \mid$ be as above. We have ${\Aut} (X) = 1$ and $m \geq 3$. We know that $(E,H^0(E))$ is a good pair if $[E] \in \mathcal S_r$ is general and that $\theta_{r,m}$ factors through the theta map $\theta_r$. Theorem 4.13 says that $\theta_{r,m}$ is generically injective if  $(E, H^0(E))$ is a good pair and the  determinant map
$$
d_E: \wedge^r H^0(E) \to H^0(\mathcal O_X(1))
$$
is injective for a general $[E]$. This is   proved in the next section.
\end{proof} 
 
 \section{ The injectivity of the determinant map} \noindent 
Let $(X,E)$ be a pair such that $X$ is a smooth irreducible curve of genus $g$ and $E$ is a semistable vector bundle of rank $r$ on $X$ and degree  $  r(g -1 +m) $, with $m \geq 3$.
If $E$ is a general semistable vector bundle on $X$, it follows that:
\begin{enumerate} \it
\item{} $(E,H^0(E))$ is a good pair,
\item{} its Pl\"ucker form exists.
\end{enumerate}
(see Definition 4.1 and Proposition 5.2). It is therefore clear that the previous conditions are satisfied on a dense open set $U$ of the moduli space of $(X,E)$. \it  \par \noindent \underline { Assumption: }\rm From now on we will assume that
$(X,E)$ defines a point of $U$, so that $X$ is a general curve of genus $g$ and $E$ is semistable and satisfies (i) and (ii). \rm  
\medskip  \par \noindent
In this section we  prove the following result:
\begin{theorem} Let $X$ and $E$ be sufficiently general and  $g >> r$, then:
\begin{enumerate} \item  the determinant map $d_E: \wedge^r H^0(E) \to H^0({\det} \ E)$ is injective, \item the classifying map $g_E: X \to \mathbb G_E$ is an embedding. 
\end{enumerate}
\end{theorem} \par \noindent
 Since $m \geq 3$, ${\rm det} \ E := \mathcal O_X(1)$ is very ample. So we  will assume as usual that the curve $X$ is embedded in $\mathbb P^n = {\mathbb P}(H^0(\mathcal O_X(1))^*)$. Let us also recall that
$$
\mathbb  G_E \subset \mathbb P^{\binom {rm}r - 1}
$$
denotes the Pl\"ucker embedding of the Grassmannian $G(r,H^0(E)^*)$. Let
$$
\lambda_E: \mathbb P^n \to \mathbb P^{\binom{rm}r - 1}
$$
be the projectivized dual of $d_E$. We have already remarked in section 4 that $g_E$ is just the restriction ${\lambda_E}_{{\vert}_ {\   X}}$. This immediately implies that
\begin{lemma} $d_E$ is injective $\Leftrightarrow$ $\lambda_E$ is surjective $\Leftrightarrow$ the curve $g_E(X)$ spans the Pl\"ucker space $\mathbb P^{\binom{rm}r -1}$.
\end{lemma} \noindent
Since $(E, H^0(E))$ is a good pair, $g_E: X \to g_E(X)$ is a birational morphism. Let
$$
\langle \ g_E(X) \ \rangle \subset \mathbb P^{ \binom {rm}r - 1}
$$
be the linear span of $g_E(X)$. Then the previous Theorem 6.1  is an immediate consequence of the following one:
\begin{theorem} For a general pair $(X,E)$ as above $g_E$ is an embedding and $${\dim} \ \langle \ g_E(X) \ \rangle    \geq \ \ r(m-1) + g.$$
 \end{theorem} \noindent
In other words, the statement says that $g_E$ is an embedding and that $d_E$ has rank $> r(m-1) + g$. This theorem and the previous lemma  imply that:
\begin{corollary} For a general  $(X,E)$,  $d_E$ is injective if $g  \geq  \binom{rm}r - r(m-1) -1$.
\end{corollary} \par \noindent
Hence the proof of Theorem 6.1 also follows.
\begin{proof}[\underline{\it Proof of Theorem 6.3}] \ \par \noindent 
To prove the theorem, hence Theorem 6.1,  we observe that the moduli space of all pairs $(X,E)$ is an integral, quasi-projective variety defined over the moduli space $\mathcal M_g$ of $X$. On the other hand, the conditions in the statement of the theorem are open. Therefore, it suffices to construct \it one pair $(X,E)$ \rm such that $E$ is semistable, $h^0(E) = rm$ and  these conditions are satisfied.
We will construct such a pair \it by induction on the genus \rm $$ g \geq 0 $$ of $X$. For $g = 0$ we have $X = \mathbb P^1$ and $E = \mathcal O_{\mathbb P^1}(m-1)^r$.
\begin{lemma} Let $X = {\mathbb P}^1$  and $E = \mathcal O_{\mathbb P^1}(m-1)^r$, with $m \geq 2$. Then $d_E$ is surjective and $g_E$ is an embedding.  \end{lemma}
\begin{proof} The proof of the surjectivity of $d_E$ is standard. It also follows from the results in  [T]. In order to deduce that $g_E$ is an embedding recall that  $g_E$ is defined by ${\rm Im} \ d_E$, hence
by the complete linear system $ \mid \mathcal O_{\mathbb P^1}(r(m-1)) \mid$. \end{proof} \noindent
Now we assume by induction  that the statement is true for $g$ and prove it for $g+1$. \par \noindent Let $(X,E)$ be a general pair such that $X$ has genus $g$. We recall that then $X$ is a general curve of genus $g$ and $(E,H^0(E))$ is a good pair  admitting  a Pl\"ucker form. \par \noindent
By induction $g_E$ is an embedding and  ${\dim} \ \langle \ g_E(X) \ \rangle \  \geq r(m - 1) + g$.  We need to prove various lemmas.
\begin{lemma} The evaluation map $ev_{x,y}: H^0(E) \to E_x \oplus E_y$ is surjective for general $x, y \in X$.
\end{lemma}
\begin{proof}  If not we would have $h^0(E(-x-y)) > h^0(E) - 2r = r(m-2)$, for any pair $(x,y) \in X^2$. This implies that $h^0(E(-x-y-z_1- \dots -z_{m-2}) ) \geq 1$, $\forall \ (x,y,z_1, \dots,z_{m-2}) \in X^m$
and hence that $(E,H^0(E))$ has no Pl\"ucker form. But then, by Proposition 5.2 (ii),  $(X,E)$ is not general: a contradiction. \end{proof} \par \noindent From now on we put $$C := g_E(X).$$ 
Choosing $x,y$ so that $ev_{x,y}$ is surjective, we have a linear embedding
$$
E^*_x \oplus E^*_y \subset H^0(E)^*
$$
induced by the dual map $ev^*_{x,y}$. This induces an inclusion of Pl\"ucker spaces
$$
\mathbb P^{ \binom {2r}r -1} := {\mathbb P} (\wedge^r (E^*_x \oplus E^*_y) ) \  \subset \  {\mathbb P}^{\binom {rm}r -1} := {\mathbb P} (\wedge^r H^0(E)^* )
$$
and of their corresponding Grassmannians
$$
\mathbb G_{x,y} := G(r, (E^*_x \oplus E^*_y)) \subset \mathbb G_E.
$$
\begin{lemma} Assume $ \langle \ C \ \rangle  $  is a proper subspace of the Pl\"ucker space of $\mathbb G_E$. Let $x,y$ be general points of $X$. Then
$\langle  \ \mathbb G_{x,y} \ \rangle  $ is not  in $\langle  \ C \ \rangle $. 
\end{lemma} 
\begin{proof}  For a general $x \in X$ consider the linear map $\pi:  H^0(E)^* \to H^0(E(-x))^*$ dual to the inclusion $H^0(E(-x)) \subset H^0(E)$. It induces a surjective linear
projection 
$$\wedge^r \pi: \mathbb P (\wedge^r H^0(E)^*) \to \mathbb P (\wedge^r H^0(E(-x))^*), $$ with center the linear span 
$ \langle \  \sigma  \ \rangle$ of $\sigma := \lbrace L \in \mathbb G_E \ \vert  \ {\rm dim} (L \cap E^*_x) \geq 1 \rbrace$. In particular $\wedge^r \pi$ restricts to a rational map between Grassmannians
$$
f: \mathbb G_E \to \mathbb G_{E(-x)},  $$
 where   $ \mathbb G_{E(-x)} := G(r, H^0(E(-x))^*)  \simeq  G(r,(m-1)r). $
Let $l \in \mathbb G_E$ be the parameter point of the space $L$, then $f(l)$ is the parameter point of $\pi(L)$.  Clearly $f$ is defined at $l$ iff $L \cap E^*_x = 0$. Moreover, the closure of the fibre of $f$ 
at $f(l)$ is the Grassmannian $G(r, L \oplus E^*_x)$.   In particular, the closure of the fibre at $f(y)$ is $\mathbb G_{x,y}$, for a general $y \in X$. 
We distinguish two cases:
\hfill \par \medskip \noindent
(1) \it $f(C)$ spans the Pl\"ucker space of $\mathbb G_{E(-x)}$. \rm  \ \  Since $f = {\wedge^r \pi}_{{\vert}_{\   \mathbb G_E}}$
 and $\wedge^r \pi$ is linear, it follows that  $\bigcup_{y \in C}  \langle \ \mathbb G_{x,y} \ \rangle $ spans the Pl\"ucker space of $\mathbb G_E$. Since $ \langle \  C \ \rangle$ is proper in it, we conclude that $\langle \  \mathbb G_{x,y}  \ \rangle$ is not in 
$\langle \ C \ \rangle $ for some $y$, hence  for general points $x, y \in X$.
\medskip \par \noindent
(2) \it $f(C)$ does not span the Pl\"ucker space of $\mathbb G_{E(-x)}$. \rm \  Since the Pl\"ucker form of $(E, H^0(E))$ exists and $m \geq 3$,   we can fix   $x, y, z_1 \dots z_{m-2} \in X$ so that $h^0(E(-x-y-z)) = 0$, where $z := z_1 + \dots + z_{m-2}$. Then we have $H^0(E(-x)) \cap H^0(E(-y-z)) = 0$ in $H^0(E)$. Putting $E^*_z := E_{z_1}^* \oplus \dots \oplus E_{z_i}^*$, it follows that $$ {\pi}_{{\vert}_{\   (E^*_z \oplus E^*_y)}}: E^*_y \oplus E^*_z \to H^0(E(-x))^*$$
is an isomorphism, that is, ${\wedge^r \pi}$ induces the following  isomorphism of projective spaces:
$$ i_{y,z} \colon {\mathbb P}({\wedge}^r( E^*_y \oplus E^*_z )) \to {\mathbb P}(\wedge^r H^0(E(-x))^*).$$
On the other hand, ${\mathbb P}({\wedge}^r( E^*_y \oplus E^*_z ))$ is spanned by the union of its natural 
linear subspaces $ \langle {\mathbb G}_{y, z_i} \rangle = {\mathbb P}({\wedge}^r( E^*_y \oplus E^*_{z_i }))$, $i = 1,..,m-2$. 
Since $ \langle \  f(C) \ \rangle$ is a proper subspace of
${\mathbb P}(\wedge^r H^0(E(-x))^*)$, it follows  that $\ \langle  \mathbb G_{y,z_i} \ \rangle$ is not in $ \langle \ C \ \rangle$,
for some $i = 1,..,m-2.$ 
 \end{proof} \par \noindent 
Now we assume that $ \langle \ C \ \rangle $ is a proper subspace of the Pl\"ucker space of $\mathbb G_E$ and  fix general points $x,y \in X$ so that the conditions of the previous lemma are satisfied. Keeping the
previous notations let $P \subset \mathbb P^{rm-1}$ be the tautological image of $\mathbb P(E^*)$ and let $P_z := \mathbb P(E^*_z)$, $z \in X$. We observe that the Grassmannian $\mathbb G_{x,y}$
is ruled by smooth rational normal curves of degree $r$ passing through $x$ and $y$. More precisely, 
let  $$\mathbb P^{2r-1} := \mathbb P (E^*_x \oplus E^*_y)$$ 
and for  $t \in \mathbb G_{x,y}$ let 
 $$ P_t \subset \mathbb P^{2r-1}   \subset \mathbb P^{rm-1}$$ 
be the projectivized space corresponding to $t$. We have:
\begin{lemma} For a general $t \in \mathbb G_{x,y}$ there exists a unique Segre product $S := \mathbb P^1 \times \mathbb P^{r-1}$ such that
$P_x \cup P_y \cup P_t \subset S \subset \mathbb P^{2r - 1}.$ Moreover: \begin{enumerate}
 \item the ruling of $S$ is parametrized by a degree $r$ rational normal curve 
$$
R \subset \mathbb G_{x,y} \subset \mathbb G_E \subset \mathbb P^{\binom {rm}r - 1},
$$
\item the universal bundle ${\mathcal U}_r$ of $\mathbb G_E$ restricts to $\mathcal O_{\mathbb P^1}(-1)^{\oplus r}$ on $R$, 
\item the restriction map $H^0({\mathcal U}^*) \to H^0(\mathcal O_{\mathbb P^1}(1))^{\oplus r})$ is surjective.
\end{enumerate}
\end{lemma}
\begin{proof} Since $x,y$ are general in $X$, Lemma 6.6 implies that $P_x \cap P_y = \emptyset$. Since $t$ is general in $\mathbb G_{x,y}$, we have  $P_t \cap P_x$ $ =$ $ P_t \cap P_y = \emptyset$. It is  a standard fact that the union of all lines in ${\mathbb P}^{2r-1}$ meeting $P_x$, $P_y$ and $P_t$ is  the Segre embedding $S \subset {\mathbb P}^{2r-1}$ of the product $\mathbb P^1 \times \mathbb P^{r-1}$,  
which  is actually the unique Segre variety containing  the  above linear spaces, see \cite{key7bis}, p.26, 2.12. 
It is also well known that 
$S$ is the tautological image of the projective bundle associated to  $\mathcal O_{\mathbb P^1}(1)^{\oplus r}$, see \cite{key7tris}. Therefore, the  map assigning  to each point $p \in {\mathbb P}^1$ the fiber of $S$ over $p$  is the classifying map of $\mathcal O_{\mathbb P^1}(1)^{\oplus r}$. So it defines an embedding of ${\mathbb P}^1$ into the Grassmannian ${\mathbb G}_{x,y}$, whose image is a rational normal curve $R$.  This implies (ii) and (iii). 
\end{proof} 
 
\noindent   Let $t \in \mathbb G_{x,y}$ be a sufficiently general point, where $x,y$ are general in $X$. Then, by  Lemma 6.7, $t$ is not in the linear space $\langle \ C \ \rangle $. Since $\mathbb G_{x,y}$ is ruled by the family of curves $R$, we can also assume that  $C \cup R$ is a nodal curve with exactly two nodes in $x$ and $y$. So far we have constructed a nodal curve
\begin{equation}
\Gamma := C \cup R
\end{equation}
such that \begin{enumerate} \it
\item $\Gamma$ has arithmetic genus $g + 1$ and degree $r(m + g)$,
\item ${\dim} \ \langle  \ \Gamma   \  \rangle  \geq \ {\dim} \ \langle  \ C \ \rangle   + 1 = r(m - 1) + g + 1$.
\end{enumerate} 
\begin{lemma} \ \begin{enumerate} \item The curve $\Gamma$ is smoothable in $\mathbb G_E$,
\item $h^1(\mathcal O_{\Gamma}(1)) = 0$ and $h^0(\mathcal O_{\Gamma}(1)) = r(m + g) - g$. \end{enumerate} \end{lemma} \noindent
Let ${\mathcal U}_r$ be the universal bundle on $\mathbb G_E$,  we have also the vector bundle on $\Gamma$: 
\begin{equation}
F := {\mathcal U}_r^* \otimes \mathcal O_{\Gamma}.
\end{equation}
\begin{lemma} \ \begin{enumerate} \item The restriction map $H^0({\mathcal U}_r^*) \to H^0(F)$ is an isomorphism, 
\item $h^1(F) = 0$ and $h^0(F) = rm$.
\end{enumerate}
\end{lemma} \par \noindent
\begin{lemma} Let $x_1, \dots , x_m$ be general points on $C$. Then $h^0(F(-x_1 - \dots - x_m)) = 0$. \end{lemma}
\begin{proof} Let us recall that $C = g_E(X)$ and that $E \cong \mathcal U^*_r \otimes \mathcal O_C$. Under the assumptions made at the beginning of this section, $X$ is a general curve of genus $g$, $(E, H^0(E))$ is a good pair admitting a Pl\"ucker form. This implies that $h^0(E(-x_1- \dots -x_m)) = 0$,  where $x_1, \dots, x_m$ are general points on $X$.  Notice also that $F \otimes \mathcal O_C \cong E$ and that, by the previous lemma, the restriction map $H^0(F) \to H^0(E)$ is an isomorphism. . \par \noindent
Let $d := x_1 + \dots + x_m$ and let $s \in H^0(F(-d))$. Then $s$ is  zero on $X$ because $h^0(E(-d)) = 0$. In particular $s$ is zero on $\lbrace x,y \rbrace = C \cap R$. Hence its restriction on $R$ is a global section $s_{\vert R}$ of $\mathcal O_R(-x-y)$. But $F \otimes \mathcal O_R(-x-y)$ is $\mathcal O_{\mathbf P^1}(-1)^{\oplus r}$ so that $s_{{\vert}_ {\ R}} = 0$. Hence $s$ is zero on $\Gamma$ and $h^0(F(-d)) = 0$. \end{proof} \par \noindent

 \par \noindent  We are now able to complete the proof of Theorem 6.3,  postponing the proofs of  lemmas 6.9 and 6.10.  \medskip \par \noindent
 \underline {\it Completion of the proof of Theorem 6.3}: \par \noindent We start from a curve $\Gamma = C \cup R$ as above. Therefore the component $C = g_E(X)$ 
 is the embedding in $\mathbb G_E$ of a curve $X$ with general moduli and, by the previous lemma, there exists $(x_1, \dots, x_m) \in C^m$ such 
 that $h^0(F(-x_1 - \dots - x_m)) = 0$. Now recall that, by lemma 6.9, the curve $\Gamma$ is smoothable in $\mathbb G_E$. This means that there exists a flat family
 $$
 \lbrace X_t, \ t \in T \rbrace
 $$
 of curves $ X_t \subset \mathbb G_E$ such that: (1) $T$ is integral and smooth, (2) for a given $o \in T$ one has $X_o = \Gamma$, (3) $X_t$ is smooth 
 for $t \neq o$. Let $$ E_t := {\mathcal U}_r^* \otimes \mathcal O_{X_t}. $$ 
For $t$ general we have  $h^1(E_t) = h^1(F) = 0$, by semicontinuity, and hence $h^0(E_t)$ $=$ $rm$. For the same reason,  the determinant map $d_t: \wedge^r H^0(E_t) \to H^0(\mathcal O_{X_t}(1))$ has rank bigger or equal to  the rank of $d_o: \wedge^r H^0(F) \to H^0(\mathcal O_{\Gamma}(1))$. This is equivalent to say that $$ {\dim} \  \langle X_t \ \rangle  \ \geq {\dim} \ \langle  \Gamma  \rangle  \ \geq \ r(m-1) + g + 1. $$
Then, for $t$ general, the pair $(X_t, E_t)$ satisfies the statement of Theorem 6.3. \par \noindent To complete the proof of the theorem, it remains to show that {\it $E_t$ is semistable for a general $t$}. It is well known that  $E_t$ is semistable if it  admits  theta divisor, see   \cite{key2}. This is equivalent to say that 
$$
\Theta_t := \lbrace N \in {\pic}^m (X_t) \  \vert  \ h^0(E_t \otimes N^{-1} ) \geq 1 \rbrace \neq {\pic}^m (X_t),
$$
therefore $E_t$ is semistable if 
$$
D_t := \lbrace (z_1, \dots, z_m) \in X_t^m \  \vert  \ h^0(E_t(-z_1- \dots -z_m)) \geq 1 \rbrace \neq X_t^m.
$$ 
To prove that $D_t \neq X_t^m$ for a general $t$, we fix in ${\mathbb G_E}^m \times T$ the family 
$$ A := \lbrace (z_1, \dots, z_m;t) \in {{\mathbb G}_E}^m \times T \  \vert  \ z_1, \dots, z_m \in X_t - {\Sing}(X_t) \  \rbrace,$$
which is integral and smooth over $T$. Then we consider its closed subset
$$ D := \lbrace (z_1, \dots, z_m; t) \in A \  \vert  \ h^0({\mathcal U}_r^* \otimes \mathcal O_{X_t}(-z_1 - \dots - z_m)) \geq 1 \rbrace. $$ 
It suffices to show that $D$ is proper, so that $D_t \neq X_t^m$ for a general $t$. Since $E_o = F$,  lemma 6.11 implies that  $D \cap X_o^m$ is proper. Indeed there exists a point $(x_1, \dots, x_m) \in C^m \subset X_o^m$ so that $h^0(F(-x_1 - \dots - x_m)) = 0$. Hence $D$ is proper. \end{proof}
\medskip \par \noindent
 \begin{proof}[\underline {Proof of Lemma 6.9}] \ \par \noindent
(i) We will put $\mathbb G := \mathbb G_E$. We recall that  $\Gamma$ is  \it smoothable \rm in $\mathbb G$ if there exists an integral variety
$ {\mathcal X} \subset \mathbb G \times T$ such that: 
\begin{enumerate} \it  \item[(a)] the projection $p: {\mathcal X} \to T$ is flat,
 \item[(b)] for some $o \in T$ the fibre  ${\mathcal X}_o$ is  $\Gamma$,
 \item[(c)] if $t \in T-\lbrace o \rbrace$, the fibre ${\mathcal X}_t$ is  smooth of genus $g+1$. 
\end{enumerate}
To prove that $\Gamma$ is smoothable we use a well known argument, see \cite{key11} or \cite{key11'}.  Consider the natural map $\phi: \mathcal T_{_{\mathbb G \vert \Gamma}} \to \mathcal N_{_{\Gamma  \vert \mathbb G}}$, where $\mathcal N_{_{\Gamma \vert \mathbb G}}$ is the normal bundle of $\Gamma$ in $\mathbb G$.  \rm The Cokernel of $ \phi$  is a sheaf $T^1_S$, supported on $S := {\Sing}(\Gamma)$. It is known as the $T^1$-sheaf of Lichtenbaum-Schlessinger. Finally, $\phi$ fits  into the following exact sequence induced by the inclusion  $\Gamma \subset \mathbb G$: 
$$ 
0 \to  \mathcal T_{_{\Gamma}} \to \mathcal T_{_{\mathbb G \mid \Gamma} } \stackrel{\phi}{\rightarrow}  \mathcal N_{_{\Gamma \vert  \mathbb G}} \to  T^1_{_{S}} \to 0.
$$
 Let $\mathcal N'$  be the image of $\phi$ in $\mathcal N_{_{\Gamma \vert \mathbb G}}$. The condition
$h^1(\mathcal N') = 0$ implies that $\Gamma$ is smoothable in $\mathbb G$, \cite{key11} prop. 1.6. 
 To show that $h^1(\mathcal N') = 0$ it is enough to show that $h^1(\mathcal T_{_{\mathbb G \vert \Gamma}})
= 0$, this is a standard argument following from the exact sequence
$$
0 \to  \mathcal T_{_{\Gamma}} \to \mathcal T_{_{\mathbb G \mid \Gamma} } \to \mathcal N' \to 0.$$
To prove that $h^1(\mathcal T_{_{\mathbb G \vert \Gamma}}) = 0$ we use the  Mayer-Vietoris exact sequence
$$
0 \to \mathcal T_{_{\mathbb G \vert \Gamma}} \to \mathcal T_{_{\mathbb G \vert C}} \oplus \mathcal T_{_{\mathbb G \vert R}} \to \mathcal T_{_{\mathbb G \vert S}} \to 0.
$$
The associated long exact yields the restriction map
$$
\rho: H^0(\mathcal T_{_{\mathbb G \vert C}}) \oplus H^0(\mathcal T_{_{\mathbb G \vert R}})\to H^0(\mathcal T_{_{\mathbb G \vert S}}).
$$
At first we show its surjectivity: it suffices to show that  $$ \rho: 0 \oplus H^0(\mathcal T_{_{\mathbb G \vert R}}) \to H^0(\mathcal T_{_{\mathbb G \vert S}})$$ is surjective. Recall that $S$ consists of two points $x,y$ and that $T^1_S = \mathcal O_S$. Then, tensoring by $\mathcal T_{_{\mathbb G \vert R}}$ the exact sequence
$$
0 \to \mathcal O_R(-x-y) \to \mathcal O_R \to \mathcal O_S \to 0,
$$
the surjectivity of $\rho$ follows if $h^1(\mathcal T_{_{\mathbb G \vert R}}(-x-y)) = 0$.  To prove this consider the standard Euler sequence defining the tangent bundle to $\mathbb G$:
$$ 0  \to {\mathcal U}_r \otimes {\mathcal U}^*_r \to {O_{\mathbb G}}^{\oplus rm} \otimes {\mathcal U}^*_r \to \mathcal T_{_{\mathbb G}} \to 0. $$
Then restrict it to $R$ and tensor by $\mathcal O_R(-x-y)$. The term in the middle of such a sequence is $M := \mathcal O_{\mathbb P^1}^{\oplus rm} \otimes \mathcal O_{\mathbb P^1}(-1)^{\oplus r}$. This just follows because $\mathcal U^*_r \otimes \mathcal O_R \cong \mathcal O_{\mathbb P^1}(1)^{\oplus r}$. Since $h^1(M) = 0$, it follows that $h^1(\mathcal T_{_{\mathbb G \vert R}}(-x-y)) = 0$. Hence $\rho$ is surjective. The surjectivity of $\rho$ and the vanishing of $h^1(\mathcal T_{_{\mathbb G \vert R}})$ and $h^1(\mathcal T_{_{\mathbb G \vert C}})$ clearly imply that $h^1(\mathcal T_{_{\mathbb G \vert \Gamma}}) = 0$. Hence we are left to show that $h^1(\mathcal T_{_{\mathbb G \vert R}}) = h^1(\mathcal T_{_{\mathbb G \vert C}}) = 0$. Since $\mathcal T_{_{\mathbb G \vert R}} \cong \mathcal O_{\mathbb P^1}^{\oplus rm} \otimes \mathcal O_{\mathbb P^1}(1)^{\oplus r}$, the former vanishing is immediate. To prove that $h^1(\mathcal T_{_{\mathbb G \vert C}}) = 0$ the argument is similar. Restricting the above Euler sequence to $C$ we obtain the exact sequence
$$0  \to E^* \otimes E \to E^{\oplus rm} \to  \mathcal T_{_{\mathbb G \vert C}} \to 0,$$
since ${{\mathcal U}^*_r}_{\vert C} \simeq  E.$ Then  $ h^1(E) =0$ implies  $h^1(\mathcal T_{_{\mathbb G \vert C}}) = 0$.  \par \noindent
(ii) To prove $h^1(\mathcal O_{\Gamma}(1)) = 0$ it suffices to consider the long exact sequence associated to the Mayer-Vietoris exact sequence
$$
0 \to \mathcal O_{\Gamma}(1) \to \mathcal O_C(1) \oplus \mathcal O_R(1) \to \mathcal O_{x, y}(1) \to 0.
$$
For degree reasons we have $h^1(\mathcal O_C(1)) = h^1(\mathcal O_R(1)) = 0$. Hence it suffices to show that the restriction  $H^0(\mathcal O_C(1)) \oplus H^0(\mathcal O_R(1)) \to \mathcal O_{x,y}$ is surjective.  This follows from the surjectivity of the restriction $H^0(\mathcal O_R(1)) \to \mathcal O_{x,y}$. \end{proof}
\begin{proof}[\underline{Proof of Lemma 6.10}] \ \par \noindent Tensoring by $F$ the standard Mayer-Vietoris exact sequence
$$
0 \to \mathcal O_{\Gamma} \to \mathcal O_C \oplus \mathcal O_R \to \mathcal O_{x, y} \to 0
$$
we have the exact sequence
$$
0 \to F \to E \oplus \mathcal O_{\mathbb P^1}(1)^{\oplus r} \to F \otimes \mathcal O_{x,y} \to 0.
$$
Passing to the associated long exact sequence we obtain
$$
0 \to H^0(F)\stackrel {u} \to H^0(E) \oplus H^0(\mathcal O_{\mathbb P^1}(1)^{\oplus r}) \stackrel {\rho} \to H^0(F \otimes \mathcal O_{x,y})\to H^1(F) \to 0
$$
Restricting $\rho$ to $H^0(E) \oplus 0$ or $0 \oplus H^0( \mathcal O_{\mathbb P^1}(1)^{\oplus r})$ we have the following maps
$$
\rho_C: H^0(E) \to E_{x} \oplus E_{y},
$$
and
$$ 
 \rho_R: H^0(\mathcal O_{\mathbb P^1}(1)^{\oplus r}) \to \mathcal O_{\mathbb P^1,x}(1)^{\oplus r} \oplus \mathcal O_{\mathbb P^1, y}(1)^{\oplus r}.
$$
These are  the usual evaluation maps and we know they are surjective. It follows from the surjectivity of $\rho$ and the above long exact sequece  that $h^0(F) = rm = h^0(\mathcal U^*_r)$ and $h^1(F) = 0$. Thus, to complete the proof, it suffices to show that  $H^0(\mathcal U_r^*) \to H^0(F)$ is injective. This is clear because  the  composition of maps $H^0(\mathcal U_r^*) \to H^0(F) \to H^0(E)$  is injective. \end{proof}

\end{document}